\documentclass[a4paper,reqno,oneside]{amsart}

\usepackage{a4wide}
\usepackage{amsmath,amsfonts,amssymb,amsthm}
\usepackage[utf8]{inputenc}
\usepackage{color}
\usepackage{graphicx}
\usepackage{enumitem}

\newtheorem{theorem}{Theorem}
\newtheorem{lemma}[theorem]{Lemma}

\newtheorem{conjecture}[theorem]{Conjecture}

\newtheorem{claim}{Claim}

\theoremstyle{definition}
\newtheorem{definition}[theorem]{Definition}

\newtheorem{problem}{Problem}

\newcommand{\oldqed}{}
\def\endofClaim{\hfill\scalebox{.6}{$\Box$}}
\newenvironment{claimproof}[1][Proof]{
  \renewcommand{\oldqed}{\qedsymbol}
  \renewcommand{\qedsymbol}{\endofClaim}
  \begin{proof}[#1]
}{
  \end{proof}
  \renewcommand{\qedsymbol}{\oldqed}
}

\newcommand{\vv}{\mathfrak{v}}
\newcommand{\bigO}{\mathcal O}

\title[On sets not belonging to algebras and rainbow matchings in graphs]{On sets not belonging to algebras and\\ rainbow matchings in graphs}

\author[D. Clemens]{Dennis Clemens}
\address{(DC) Technische Universität Hamburg-Harburg, Institut f\"ur Mathematik, Am Schwarzenberg-Campus 3, 21073 Hamburg, Germany }
\email{dennis.clemens@tuhh.de}
\author[J. Ehrenm\"uller]{Julia Ehrenm\"uller}
\address{(JE) Technische Universität Hamburg-Harburg, Institut f\"ur Mathematik, Am Schwarzenberg-Campus 3, 21073 Hamburg, Germany }
\email{julia.ehrenmueller@tuhh.de}
\author[A. Pokrovskiy]{Alexey Pokrovskiy}
\address{(AP) Freie Universität Berlin, Methods for Discrete Structures, Arnimallee 3, 14195 Berlin, Germany }
\email{alja123@gmail.com}

\begin{document}
\begin{abstract}
Motivated by a question of Grinblat, we study the minimal number $\vv(n)$ that satisfies the following. If $A_1,\ldots, A_n$ are equivalence relations on a set $X$ such that for every $i\in[n]$ there are at least $\vv(n)$ elements whose equivalence classes with respect to $A_i$ are nontrivial, then $A_1, \ldots, A_n$ contain a rainbow matching, i.e.~there exist $2n$ distinct elements $x_1,y_1,\ldots,x_n,y_n\in X$ with $x_i\sim_{A_i} y_i$ for each $i\in [n]$. 
Grinblat asked whether $\vv(n) = 3n-2$ for every $n\geq 4$. The best-known upper bound was $\vv(n) \leq 16n/5 + \bigO(1)$ due to Nivash and Omri. Transferring the problem into the setting of edge-coloured multigraphs, we affirm Grinblat's question asymptotically, i.e.~we show that $\vv(n) = 3n+o(n)$.

%
%

\end{abstract}
\maketitle
\section{Introduction}

Let $X$ be a set and let $\mathcal P(X)$ denote its power set. 
A nonempty subset $\mathcal A \subseteq \mathcal P(X)$ is an \emph{algebra} on $X$ if 
$\mathcal A$ is closed under complementation and under unions, i.e.~if 
$M_1, M_2 \in \mathcal A$, then $X \setminus M_1 \in \mathcal A$ and $M_1 \cup M_2 \in \mathcal A$. 
In a series of papers and books~\cite{grinblat2002algebras,grinblat2004theorems,grinblat2015families}
Grinblat  investigated sufficient conditions for countable families $\{\mathcal A_i\}_i$  
of algebras such that $\bigcup_i A_i \neq \mathcal P(X)$ and $\bigcup_i A_i = \mathcal P(X)$, respectively. 
In this context, Grinblat \cite{grinblat2002algebras} defined $\vv=\vv(n)$ as the minimal cardinal number such that the following is true. 
``Let $\mathcal A_1, \ldots, \mathcal A_n$ be algebras on a set $X$ such that for each $i\in [n]$ there exist at least 
$\vv(n)$ pairwise disjoint sets in $\mathcal P(X) \setminus \mathcal A_i$. Then there exists a family $\{U^1_i, U^2_i\}_{i\in [n]}$ 
of $2n$ pairwise disjoint subsets of $X$ such that, for each $i\in [n]$, if $Q \in \mathcal P(X)$ and $Q$ contains one of the two sets 
$U^1_i$ and $U^2_i$ and its intersection with the other one is empty, then $Q \notin \mathcal A_{i}$.''
In~\cite{grinblat2002algebras} Grinblat showed that $\vv(n) \geq 3n-2$ for each $n\in \mathbb N$.
 He posed the following problem. 
\begin{problem}[Grinblat, \cite{grinblat2015families}]\label{GrinblatProblem}
Is it true that $\vv(n)=3n-2$ for  $n\geq 4$?
\end{problem}

In~\cite{grinblat2015families} Grinblat proved 
that $\vv(n) \leq 10n/3 + \sqrt{2n/3}$. Nivasch and Omri~\cite{nivasch2015rainbow} improved the upper bound to $\vv(n) \leq 16n/5 + \bigO(1)$ using the following 
equivalent definition of $\vv(n)$ in the context of equivalence relations. 
Let $X$ be a finite set and let $A$ be an equivalence relation on $X$. 
If $x,y \in X$ are equivalent  under $A$, we write $x\sim_A y$. 
With \[ [x]_A = \{y \in X: x \sim_A y\}\]
we denote the \emph{equivalence class} under $A$ of an element $x\in X$,  
while the \emph{kernel} of $A$ is defined as \[\ker(A)= \{x\in X: |[x]_A| \geq 2\}.\]
Using these definitions, it turns out that $\vv(n)=\vv_1(n)$ holds, where $\vv_1(n)$ is defined to be the minimal number such that if $A_1, \ldots, A_n$ are equivalence relations with $\ker(A) \geq \vv_1(n)$ for each $i\in [n]$, then $A_1, \ldots, A_n$ contain a \emph{rainbow matching}, i.e. a set of $2n$ distinct elements $x_1,y_1, \ldots, x_n, y_n \in  X$ with $x_i \sim_{A_i} y_i$ for each $i\in [n]$. This identity is mainly based on the fact that there is a natural correspondence between algebras and equivalence relations. Indeed, given an equivalence relation $A$ on the set $X$, we can define the algebra ${\mathcal A}:=\left\{\bigcup_{x\in S}[x]_{A}:\ S\subseteq X\right\}$. Conversely, given some algebra $\mathcal A$ on $X$, one can define the equivalence relation $A$ on $X$ the equivalence classes of which are the inclusion minimal sets in $\mathcal A$. A complete argument to show that $\vv(n)=\vv_1(n)$ is given in the appendix.  


In this paper we show that $\vv(n)\leq 3n+o(n)$, thus giving an asymptotic answer to Problem~\ref{GrinblatProblem}. More precisely, using the terminology of Nivasch and Omri~\cite{nivasch2015rainbow}, we prove the following.  

\begin{theorem}\label{thm:main1}
For every $\delta >0$ there exists $n_0= n_0(\delta) = 144/\delta^2$ such that the following holds for every $n\geq n_0$. 
Let $A_1,\ldots,A_n$ be $n$ equivalence relations on a finite set $X$. If $|\ker(A_i)| \geq (3+\delta)n$ for each $i\in[n]$, 
then $A_1,\ldots,A_n$ contain a rainbow matching. 
\end{theorem}

Theorem~\ref{thm:main1} can be rephrased in the context of graphs. If $A_1, \ldots, A_n$ are equivalence relations on a set $X$, let the vertices of an edge-coloured multigraph be the elements of $X$ and,  for each $i \in [n]$, let $\{x,y\}\in \binom{X}{2}$ be an edge of colour $i$ if and only if $x\sim_{A_i} y$. This means that the equivalence relations are represented in this multigraph by colour classes, each of which is the disjoint union of nontrivial cliques, i.e.~complete graphs with at least 2 vertices. A matching in an edge-coloured multigraph is called \emph{rainbow matching} if all its edges have distinct colours. Using this notion, we can reformulate Theorem~\ref{thm:main1} as follows. 

\begin{theorem}\label{thm:main2}
For every $\delta >0$ there exists $n_0= n_0(\delta) = 144/\delta^2$ such that the following holds for every $n\geq n_0$. 
Let $G$ be a multigraph, the edges of which are coloured with $n$ colours 
and each subgraph of which induced by a colour class has at least $(3+\delta)n$ vertices and is the disjoint union of nontrivial cliques. 
Then $G$ contains a rainbow matching of size $n$.  
\end{theorem}

Theorem \ref{thm:main2} is a strengthening of an earlier result by the first two authors \cite{clemens2015}, which proves the above statement when the multigraph $G$ is bipartite (and thus each clique consists of two vertices). The latter was motivated by famous conjectures of Ryser~\cite{ryser1967} and of Brualdi and Stein \cite{brualdi1991,stein1975} on Latin squares and by the following conjecture of Aharoni and Berger \cite{aharoni2009}.

\begin{conjecture}[Aharoni and Berger \cite{aharoni2009}]
Let $G$ be a multigraph, the edges of which are coloured with $n$ colours 
and such that each colour class induces a matching of size $n+1$. 
Then there is a rainbow matching of size $n$.
\end{conjecture}

These conjectures remain widely open. However, asymptotic versions are known to be true. 
For instance, as a consequence of a theorem of H\"aggkvist and Johansson \cite{haggkvist2008} one obtains that there is a rainbow matching of size $n$ in case when $G$ is an edge-coloured bipartite graph the   
colour classes of which induce perfect matchings of size $n+o(n)$. The third author \cite{pokrovskiy2015} provided a proof for the more general case where the matchings are disjoint, but not necessarily  perfect.

In the next section we prove Theorem~\ref{thm:main2} which automatically provides a proof for Theorem~\ref{thm:main1}.  
As already mentioned, the best-known lower bound on $\vv(n)$ is $3n-2$ for each $n\geq 4$. Indeed, if all colour classes are identical and are the disjoint union of $n-1$ triangles, then there is no rainbow matching of size $n$. Hence, Theorem~\ref{thm:main2} is asymptotically best possible. 
If $n=3$, then $\vv(3)=9 > 3n-2$ as shown by Grinblat~\cite{grinblat2002algebras}. See Figure~\ref{fig:n3} for the lower bound $\vv(3) \geq 9$, which was also observed by Nivasch and Omri~\cite{nivasch2015rainbow}.

\begin{figure}[h]
\begin{center}
\includegraphics[scale=0.7]{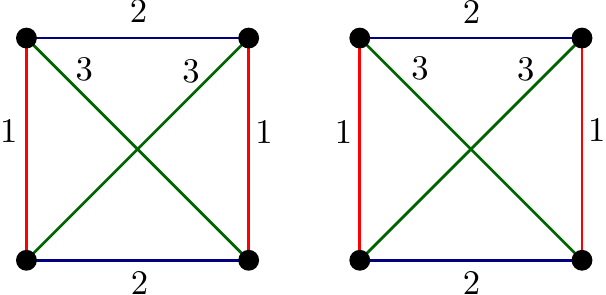} 
\caption{Example of a graph with $3$ colour classes each of size $8$ that has no rainbow matching of size $3$. }
\label{fig:n3}
\end{center}
\end{figure}

\section{Proof of Theorem~\ref{thm:main2}}

The aim of this section is to present the proof of Theorem~\ref{thm:main2}. 
We will start with stating and proving a lemma, which will be the essential part in the proof of Theorem~\ref{thm:main2}.  
However, we first need to introduce some necessary definitions and notation. 

Given a multigraph $G$, we denote by $E_G[A,B]$ the set of edges in $G$ that have one endvertex in $A \subseteq V$ and one endvertex in $B\subseteq V$. 
For any edge-coloured multigraph $G$, we denote by $c(e)$ the colour assigned to the edge $e \in E(G)$. 
For the sake of simplicity, we call an edge of colour $c$ simply \emph{$c$-edge}. 
Let $F$ be a set or a sequence of edges, then $V(F):= \bigcup_{e\in F} e$ is the \emph{vertex set} of $F$.  
Next we define switchings, which, given some rainbow matching $M$ of size $k$, provide us a new rainbow matching of size $k$ by replacing edges in $M$ with edges in $E(G)\setminus M$. See Figure~\ref{fig:switching} for an illustration of a switching of length $3$. 

\begin{definition}
Let $G$ be an edge-coloured multigraph and let $M$ be a rainbow matching in $G$. 
We call a sequence of edges $\sigma = (e_0, m_1, e_1, m_2, \ldots, e_{k-1}, m_k)$ 
a \emph{$\big(c(e_0), c(m_k)\big)$-switching} of length $k$ with respect to $M$  if for each $i\neq j\in[k]$ we have 
\begin{enumerate}
\item[(S1)] $m_1, \dots, m_k$ are distinct edges in  $M$,
\item[(S2)] 
$e_{i-1} \in E_G\big[m_i, V\setminus V(M)\big]$, 
\item[(S3)] $c(e_0) \neq c(m_i)$ and $c(e_i) = c(m_i)$, and 
\item[(S4)] $e_{i-1} \cap e_{j-1} = \varnothing$. 
\end{enumerate}
Whenever it is clear from the context, we omit writing with respect to which matching a switching is defined. The length of $\sigma$ is denoted by $\ell(\sigma)$. Furthermore, we denote by $m(\sigma)$ the set of all edges of $\sigma$ that are 
contained in the matching $M$ and by $e(\sigma)$ the set of all other edges of $\sigma$. Observe that $\ell(\sigma) = |m(\sigma)| = |e(\sigma)|$.

For every colour $c$, we also define an \emph{empty} $(c,c)$-switching $\sigma_c^0$. This switching has no edges, starts and ends at the colour $c$, has length zero, and has $m(\sigma_c^0)=e(\sigma_c^0)=\emptyset$.
\end{definition}

\begin{figure}[h]
\begin{center}
\includegraphics[scale=0.7]{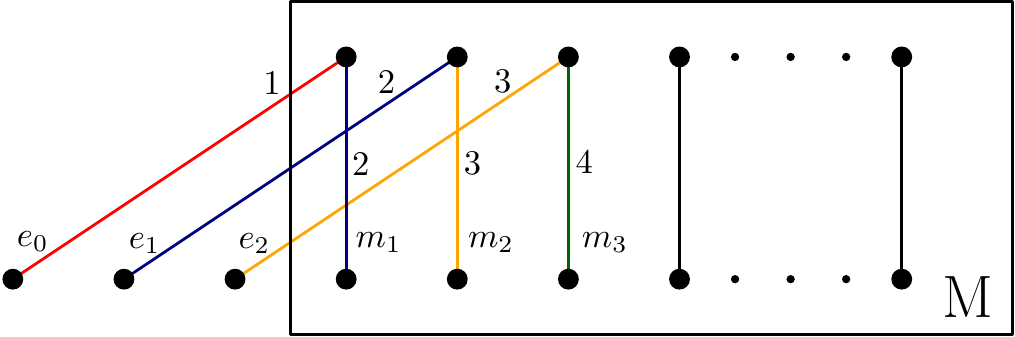} 
\caption{A $(1,4)$-switching of length $3$. }
\label{fig:switching}
\end{center}
\end{figure}

Now we are in the position to state and prove Lemma~\ref{thm:main3} which is the main technical result of the paper. 

\begin{lemma}\label{thm:main3}
For each $n\in \mathbb{N}$ and $\delta >0$  satisfying $\delta\sqrt{n} \geq 12$, the following holds.
Let $G=(V,E)$ be a multigraph whose edges are coloured by $n$ colours, $M$ a rainbow matching of size $n-1$ in $G$, and $c_0$ the colour that is missing from $M$. 

Suppose that for every colour $c$ in $G$, and every $(c_0,c)$-switching $\sigma$ there are at least $\big(\lceil(1+\delta)n\rceil - 4\ell(\sigma)$ disjoint $c$-edges between $V\setminus \big(V(M) \cup V(\sigma)\big)$ and $V\setminus V(\sigma)$.  
Then $G$ has a rainbow matching of size $n$. 
\end{lemma}
An important special case of the condition in Lemma~\ref{thm:main3} is when $c=c_0$ and $\sigma$ is the empty switching $\sigma_{c_0}^0$. In this case the condition says that there are at least $\lceil(1+\delta)n\rceil$ disjoint $c_0$-edges touching  $V\setminus V(M)$.

\begin{proof}[Proof of Lemma~\ref{thm:main3}]
Let $C$ be the set of colours of edges of $G$ and $R:= V\setminus V(M)$.  We prove Lemma~\ref{thm:main3} by induction on $n$. 
For the initial case, we prove the theorem for all $n\leq 144$.
Notice that if $n\leq 144$, then from $\delta\sqrt{n} \geq 12$, we obtain  $\delta\geq 1$. This means in particular that there are $2n$ disjoint edges of colour $c_0$ in $E_G[R, V]$. However, there can be at most $|V(M)|=2n-2$ disjoint $c_0$-edges in $E_G[R,V(M)]$. Hence, there exists a $c_0$-edge in $E_G[R,R]$ which can be added to $M$ in order to obtain a rainbow matching of size $n$. 

Now let $n > 144$ and assume that Lemma~\ref{thm:main3} holds for every $n'< n$. We may also assume  that $\delta \leq 1$ since otherwise there is a rainbow matching of size $n$ by the same argument as before. 
Let $G=(V,E)$ be a multigraph and $M$ a rainbow matching of size $n-1$ in $G$, which satisfies all the assumptions of the lemma. Suppose for the sake of contradiction that $G$ does not have a rainbow matching of size $n$.
The following claim produces a switching, a set of colours, and a set of edges that will later be used to reduce the problem to a smaller multigraph, to which we apply induction. 

\begin{claim}\label{claim:first}
There exist a colour $c_2\in C$, a $(c_0,c_2)$-switching $\sigma = (e_0, m_1, e_1, m_2)$ and 
a subset $C^\ast \subseteq C\setminus \{c_0,c_1,c_2\}$, where $c_1:= c(m_1)$, with $|C^\ast| = \lceil\delta n/6\rceil$, such that for each $c \in C^\ast$ there exists a $c$-edge $e_c$ between $V\setminus \big(V(M)\cup V(\sigma)\big)$ and $m_2\setminus e_1$. 
\end{claim}

In Figure~\ref{fig:claim1} we illustrate the switching, the set of colours and the edges that are guaranteed by Claim~\ref{claim:first}. 

\begin{figure}[h]
\begin{center}
\includegraphics[scale=0.7]{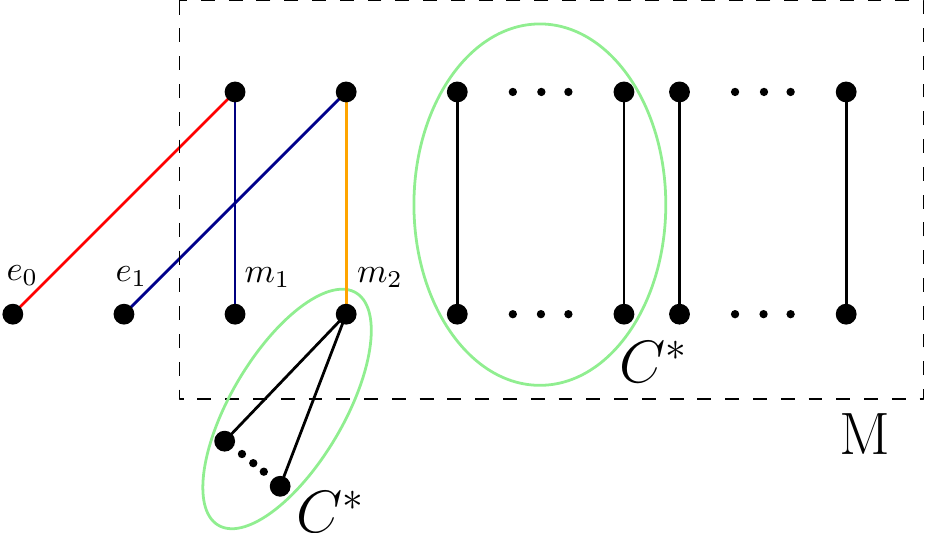} 
\caption{Switching $(e_0,m_1,e_1,m_2)$, set $C^\ast\subseteq C\setminus \{c_0,c_1,c_2\}$ and $c$-edges in $V\setminus \big(V(M)\cup V(\sigma)\big)$ with $c\in C^\ast$. }
\label{fig:claim1}
\end{center}
\end{figure}

\begin{claimproof}[Proof of Claim~\ref{claim:first}]
Let $C_1:= \big\{c\in C: \exists (c_0,c)\text{-switching of length }1\big\}$. 
By the assumption of the lemma, there exist $\lceil(1+\delta)n\rceil$ disjoint $c_0$-edges having an endvertex in $R$. 
If there exists a $c_0$-edge $e\in E_G[R,R]$, then $M \cup \{e\}$ is a rainbow matching of size $n$. 
Therefore we may assume that all $c_0$-edges from $R$ end in $V(M)$, which implies that $|C_1|\geq \lceil(1+\delta)n\rceil/2$. 

For every $c\in C_1$, let $\sigma_c=(e^c_0, m_1^c)$ be an arbitrary but fixed $(c_0,c)$-switching of length $1$. 
For a colour $c\in C_1$, we say that an edge $m\in M$ is $c$-good if there exist two disjoint $c$-edges in $E_G\big[R\setminus V(\sigma_c), m\big]$. 
By the assumption of the lemma, for every $c\in C_1$, there exist $\lceil(1+\delta)n\rceil-4$ disjoint $c$-edges in $E_G\big[R\setminus V(\sigma_c),V\setminus V(\sigma_c)\big]$.
If there exists a colour $c\in C_1$ and a $c$-edge $e \in E_G\big[R\setminus V(\sigma_c), R\setminus V(\sigma_c)\big]$, then there is a rainbow matching of size $n$, namely the union of  
the subset of $M$ induced by the colours in $C\setminus \{c\}$, and the edges $e$ and $e_0^c$.  
Therefore we may assume that for every $c\in C_1$, there exist $\lceil(1+\delta)n\rceil-4$ disjoint $c$-edges in $E_G\big[R\setminus V(\sigma_c), V(M)\setminus V(\sigma_c)\big]$. 

The maximum number of disjoint $c$-edges  in $E_G\big[R\setminus V(\sigma_c), V(M)\setminus V(\sigma_c)\big]$ is less than twice the number of $c$-good edges $m\in M$ plus the number of edges $m\in M$ which are not $c$-good.
Hence, for every $c\in C_1$, there exist at least $\lceil\delta n\rceil-2$ edges in $M\setminus m(\sigma_c)$ that are $c$-good. Next we find an edge $m$ which is $c$-good  for many colours $c\in C_1$.
Let 
\[\mu:= \max_{m\in M} \{|C'|: C' \subseteq C_1\setminus \{c(m)\} \text{ such that } m \text{ is } c\text{-good for each } c \in C'\}.\]
Double counting the pairs $(c,m)$, where $c\in C_1\setminus \{c(m)\}$ and $m$ is a $c$-good edge, yields
\[\mu |M| \geq |C_1|\big(\lceil\delta n\rceil -2\big)\]
and hence
\[\mu \geq \frac{(\delta n-2)(1+\delta)}{2}.\]
This means that  there exists an edge $m_2=\{x,y\} \in M$ and a subset $C_2 \subseteq C_1\setminus\{c(m_2)\}$ of size $\lceil(\delta n -2)(1+\delta)/2\rceil$ such that $m_2$ is $c$-good for every $c\in C_2$. 
For every $c\in C_2$, let $x_c, y_c \in E_G\big[R\setminus V(\sigma_c),m_2\big]$ be disjoint edges of colour $c$ such that $x_c\cap m _2= \{x\}$ and $y_c\cap m_2 = \{y\}$ (such edges exist since $m_2$ is $c$-good). Let $X:= \{x_c: c\in C_2\}$ and $Y:=\{y_c: c\in C_2\}$. 
The remainder of the proof is split into the case that there exists a vertex in $R$ that is incident to at least $1/3$ of the edges in $X$ and the case that there does not exist such a vertex. 

First suppose that  there exists a vertex $v\in R$ such that $v$ is incident to at least $1/3$ of the edges in $X$.
Using $\delta \sqrt{n} \geq 12$ and $\delta\leq 1$, notice that $|X|/3=|C_2|/3\geq  (\delta n -2)(1+\delta)/6\geq \lceil\delta n/6\rceil+1$.
Therefore, we can let $X'$ be a subset of $X$ consisting of $\lceil\delta n/6\rceil+1$ edges such that $v \in e$ for every $e\in X'$. Let $e_1'$ be any edge in $X'$. 
Since $c(e_1')\in C_2$, there is an edge $e_1 \in Y$ with $c(e_1') = c(e_1)$. 
By the definition of $X$ and $Y$, we also have $e_1\cap e_1' = \varnothing$ and $V(\sigma_{c(e_1)}) \cap e_1' = \varnothing$, which imply $X' \subseteq E_G\big[R\setminus \big(V(\sigma_{c(e_1)})\cup e_1\big), x\big]$. 
Set $c_1:= c(e_1)$, $c_2:=c(m_2)$, $e_0:=e_0^{c_1}$, and $m_1:= m_1^{c_1}$. 
We show that the set $C^\ast:=  \{c\in C_2: x_c\in X'\setminus\{e_1'\} \}$, the sequence $\sigma:=(e_0,m_1,e_1,m_2)$, and edges $e_c:=x_c$ for each $c\in C^\ast$ are as desired in the claim. First let us argue that $\sigma$ is indeed a $(c_0,c_2)$-switching. 
Property~(S1) is fulfilled since $m_1 \in M$ as $(e_0,m_1)$ is a switching and since $m_2 \in M$ by the choice of $m_2$. Property~(S2) holds since $(e_0,m_1)$ is a switching and since $e_1\cap m_2 = \{y\} \neq \varnothing$ and $e_1 \cap R \neq \varnothing$ by the definition of $Y$. 
As $c_0$ is not assigned to edges in $M$, we have $c(m_1)\neq c(e_0)\neq c(m_2)$. 
Moreover, we have $c(m_1) = c(e_1)$ by construction. 
This shows Property~(S3). 
Finally, Property~(S4) is satisfied since we have $e_1 \in E_G\big[R\setminus V(\sigma_{c_1}), m_2\big]$ by definition of $e_1\in Y$, and hence $e_0\cap e_1 = \varnothing$. 
Thus, $\sigma$ is indeed a $(c_0,c_2)$-switching. 
Observe that $C^\ast \subseteq C_2\setminus \{c_1\} \subseteq C_1\setminus \{c_1,c_2\} \subseteq C\setminus \{c_0,c_1,c_2\}$. 
Finally, for each $c\in C^\ast \subseteq C$, we have $e_c= x_c \in E_G\big[V\setminus\big(V(M)\cup V(\sigma)\big), x\big] = E_G\big[V\setminus\big(V(M)\cup V(\sigma)\big), m_2\setminus e_1\big]$.

Now suppose that all vertices in $R$ are incident to at most $1/3$ of the edges in $X$. 
Let $e_1$ be any edge in $Y$. Then at least $1/3$ of the edges in $X$ are disjoint from $e_1$ and $\sigma_{c(e_1)}$. Since, as before, $|X|/3\geq (\delta n -2)(1+\delta)/6\geq \lceil\delta n/6\rceil+1$ we can choose a subset $X^\ast\subseteq X$ of size $\lceil\delta n/6\rceil$ such that for every $e\in X^\ast$ we have $c(e) \neq c(e_1)$ and $e\cap \big(e_1\cup \sigma_{c(e_1)}\big) = \varnothing$. 
Set again $c_1:=c(e_1)$, $c_2:= c(m_2)$, $e_0:= e_0^{c_1}$, and $m_1:=m_1^{c_1}$. Analogously to the previous case, the set $C^\ast := \{c\in C_2: x_c\in X^\ast\}$, the sequence $\sigma:= (e_0,m_1,e_1,m_2)$, and the edges $e_c:=x_c$ for each $c\in C^\ast$ are as desired in the claim. 
\end{claimproof}

From now on we may assume the existence of the switching $\sigma=(e_0,m_1,e_1,m_2)$, the set $C^\ast\subseteq C$ and the edges
$e_c$ as in Claim~\ref{claim:first}. We consider the following sets.
Let 
\begin{align*}
W &:= \{e\in M:\ c(e)\in C^\ast\},\\
M' &:= M\setminus (m(\sigma)\cup W),\\
C'&:= C\setminus (C^\ast \cup \{c_0,c_1\} ).
\end{align*}
Observe that $C'$ is exactly the set of all colours assigned to the edges of $M'$ plus colour $c_2$. 
Note further that $|W|=|C^\ast|=\lceil\delta n/6\rceil$. Moreover, we set
\begin{align*}
n' &:=|C'|=\lfloor n(1-\delta/6)\rfloor -2,\\
S &:= \bigcup_{c\in C^\ast}e_c\cap R,\\
R' &:= R\setminus (V(\sigma) \cup S).
\end{align*}

See Figure~\ref{fig:sets} for an illustration of the sets $M', W \subset M$, the set $S \subset R$ and the switching $(e_0,m_1,e_1,m_2)$.

\begin{figure}[h]
\begin{center}
\includegraphics[scale=0.7]{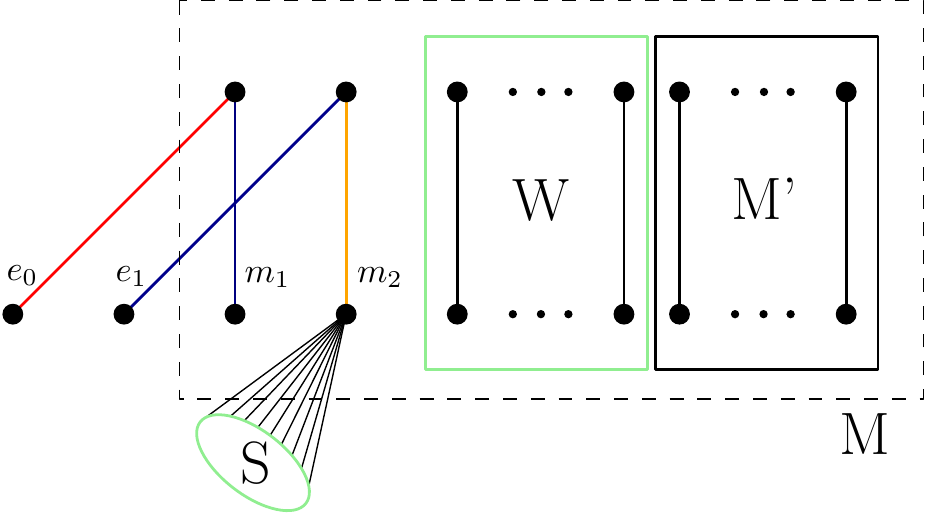} 
\caption{Sets $M', W \subset M$, set $S \subset R$ and switching $(e_0,m_1,e_1,m_2)$. }
\label{fig:sets}
\end{center}
\end{figure}

To apply induction, we now consider the edge-coloured multigraph $G'$ formed from $G$ by deleting edges of colours from $C^\ast\cup \{c_0, c_1\}$ and vertices in $\big(V(\sigma)\cup S\cup V(W)\big)$. Formally, let $G'=(V',E')$ be the multigraph with vertex set
\[V':=V\setminus (V(\sigma)\cup S\cup V(W))=R'\cup V(M')\] and edge set
\[E':=\{e\in E_G[V']:\ c(e)\in C'\}.\] The edges of $G'$ keep the colours they had in $G$.

With this notation in hand, we prove the following claim.

\begin{claim}\label{claim:second}
There is a constant  $\delta'\geq 12/\sqrt{n'}$ such that 
 for every colour $c\in C'$  and $(c_2,c)$-switching $\sigma$ (in $G'$, with respect to $M'$) there are at least $\big(\lceil(1+\delta')n'\rceil-4\ell(\sigma)$ disjoint $c$-edges between
 $V'\setminus \big(V(M') \cup V(\sigma')\big)$ and $V'\setminus V(\sigma')$.
\end{claim}

\begin{claimproof}[Proof of Claim~\ref{claim:second}]
Set $\delta'=(\lceil\delta n\rceil - 12)/n'$ and notice that the following holds:
$$
\delta'n'\geq \delta n -12\geq 12(\sqrt{n}-1)\geq 12\sqrt n(1-\delta/12)\geq 12\sqrt{n}\sqrt{1-\delta/6} > 12\sqrt{n'}.
$$
The second and third inequalities use $\delta\geq 12/\sqrt{n}$. Therefore $\delta'\geq 12/\sqrt{n'}$ holds, and so we can apply induction to the multigraph $G'$ with the matching $M'$.
Consider any colour $c\in C'$ and let $\sigma'$ be some $(c_2,c)$-switching in $G'$.
We need to show that
are at least $\big(\lceil(1+\delta')n'\rceil-4\ell(\sigma')\big)$ disjoint $c$-edges in $E_{G'}\big[R'\setminus V(\sigma'),V'\setminus V(\sigma')\big]$. 
Recall that $\sigma$ is the $(c_0,c_2)$-switching, given by Claim~\ref{claim:first}.
Then the concatenation of $\sigma$ and $\sigma'$ gives a $(c_0,c)$-switching $\sigma''$ (in $G$ w.r.t.~$M$) of length $\ell(\sigma')+2$. So, by the assumption of Lemma~\ref{thm:main3} on $G$, we can find 
at least $\lceil(1+\delta)n\rceil-4\ell(\sigma'')= \lceil(1+\delta)n\rceil-4(\ell(\sigma')+2)$ disjoint $c$-edges in $E_G\big[R\setminus V(\sigma''),V\setminus V(\sigma'')\big]$. 
As $|R\setminus V(\sigma'')|-|R'\setminus V(\sigma'')|=|S|\leq \lceil\delta n/6\rceil$, 
at least $\lceil(1+\delta)n\rceil - \lceil\delta n/6\rceil-4(\ell(\sigma')+2)$ of these disjoint edges
belong to $E_G\big[R'\setminus V(\sigma''),V\setminus V(\sigma'')\big]\subseteq E_G\big[R'\setminus V(\sigma'),V\setminus V(\sigma')\big]$. 

Assume first that there is no edge $e\in E_G\big[R'\setminus V(\sigma'),V(W)\cup S\big]$ with $c(e)=c$.
Then, since at most 6 disjoint edges intersect $V(\sigma)$, the claim holds since the number of $c$-edges 
in $E_{G'}\big[R'\setminus V(\sigma'),V'\setminus V(\sigma')\big]$ is at least
\begin{align*}
\lceil\left(1+\delta\right)n\rceil -\left\lceil\tfrac{\delta n}{6}\right\rceil - 4(\ell(\sigma')+2)-6
= (1+\delta')n'-4\ell(\sigma').
\end{align*}

Assume then that there is an edge $e\in E_G\big[R'\setminus V(\sigma'),V(W)\cup S\big]$ with $c(e)=c$.
If $e\cap S\neq \varnothing$, then $(M\setminus m(\sigma''))\cup (\{e\}\cup e(\sigma''))$ is a rainbow matching of size $n$ in $G$. Otherwise, if $e\cap V(W)\neq \varnothing$,
then let $f\in W$ with $f\cap e\neq \varnothing$. By definition of $W$ and $S$, and by Claim \ref{claim:first}
we find an edge $g\in E_G(S,m_2\setminus e_1)$
with $c(g)=c(f)$. Then $(M\setminus (m(\sigma'')\cup \{f\}))\cup (e(\sigma'')\cup \{e,g\})$ is a rainbow matching of size $n$ in $G$, contradicting our assumption that $M$ was maximum.
\end{claimproof}

Now we are able to finish the induction. By Claim~\ref{claim:second}, $G'$ satisfies the hypothesis of Lemma~\ref{thm:main3}. Therefore, since $n'<n$, by induction $G'$ contains a rainbow matching $M''$ of size $n'$.
Now, $M''\cup W\cup e(\sigma)$ forms a rainbow matching of size $n$, contradicting our assumption there were no rainbow matchings in $G$ of this size.
\end{proof}

Finally, we are ready to prove Theorem~\ref{thm:main2}. 

\begin{proof}[Proof of Theorem~\ref{thm:main2}]
Let $\delta>0$ and $n\geq 144/\delta^2$ and let $G$ be given according to the theorem. For the sake of  contradiction, let us assume that a largest rainbow matching $M$ in $G$ has size smaller than $n$. Let $C'$ be the set of colours in $M$ plus one further colour $c_0$. Set $n':=|C'|$. In the following, we consider the multigraph $G'=(V,E')$ with $E'=\{e\in E: c(e)\in C'\}$. We now apply Lemma~\ref{thm:main3} to $G'$ in order to find a rainbow matching of size $n'$.  This gives a contradiction since we assumed  that $M$  was  a maximum matching.

Let $\delta'=\delta n/n'$
and observe that from $n\geq 144/\delta ^2$, we have $\delta'\geq 12/\sqrt{n'}$. Let $c\in C'$ and let $\sigma$ be any $(c_0,c)$-switching in $G'$ with respect to $M$. 
By assumption on $G$,
there exist $\lceil(3+\delta)n\rceil-|V(M)\cup V(\sigma)|>\lceil(1+\delta)n\rceil-\ell(\sigma)>\lceil(1+\delta')n'\rceil-\ell(\sigma)$
vertices in $V\setminus \big(V(M)\cup V(\sigma)\big)$ that are incident to colour $c$.
If in the colour class of $c$ any two of these vertices are adjacent or have a common neighbour, then since all the colour classes in $G$ are unions of cliques, 
there is an edge, say $e$, of colour $c$ between them, which leads to the rainbow matching
$(M\setminus m(\sigma))\cup (e(\sigma)\cup \{e\})$ of size $n'$. So, we may assume that there are
$\big(\lceil(1+\delta')n'\rceil-\ell(\sigma)\big)$ disjoint edges of colour $c$ in $E_{G'}\big[V\setminus \big(V(M)\cup V(\sigma)\big),V\big]$
 and therefore $\big(\lceil(1+\delta')n'\rceil-4\ell(\sigma)\big)$ such edges in $E_{G'}\big[V\setminus \big(V(M)\cup V(\sigma)\big),V\setminus V(\sigma)\big]$. As $c$ and $\sigma$ were chosen arbitrarily, Lemma~\ref{thm:main3} now guarantees that $G'$ contains a rainbow matching of size $n'$.  
\end{proof}

\section{Concluding Remarks}

We wonder how the problem changes if one adds the natural constraint that every pair of distinct elements belongs to at most one equivalence relation. More precisely, we are interested in the following problem. 

\begin{problem}
Determine the minimal number $\vv^\ast(n)$ such that if $A_1,\ldots, A_n$ are equivalence relations on a set $X$ with $|\ker(A_i)| \geq \vv^\ast(n)$ and $A_i\cap A_j = \big\{(x,x):x\in X\big\}$ for each $i\neq j\in [n]$,  
then $A_1, \ldots, A_n$ contain a rainbow matching.  
\end{problem}

Using the graph theoretic notion as before, the additional constraint means that the colour classes need to be pairwise disjoint. This can also be seen as restricting the problem to graphs instead of multigraphs. It is known that for every even $n$, there exists an edge-coloured bipartite graph, the colour classes of which are matchings of size $n$, without a rainbow matching of size $n$. This follows from a result by Euler on transversals in Latin squares (see e.g.~\cite[page 263]{wilson2013}). For general $n$ we thus obtain $\vv^\ast(n) > 2n-2$. An upper bound on $\vv^\ast(n)$ follows directly from Theorem~\ref{thm:main2}, i.e.~$\vv^\ast(n) \leq \vv(n) = 3n + o(n)$.  

\bibliographystyle{abbrv} 	 
\bibliography{RainbowMatchings}

\begin{appendix}
\section{}
Here we show that two versions of the problem studied in this paper are equivalent. Specifically we show that $\vv(n)= \vv_1(n)$, where $\vv(n)$ and $\vv_1(n)$ are as defined in the introduction.

First we show that $\vv(n)\geq \vv_1(n)$ holds. Let $A_1,\ldots,A_n$ be equivalence relations with $|\ker(A_i)|\geq \vv(n)$ for each $i\in [n]$. Let ${\mathcal A}_i:=\left\{\bigcup_{x\in S}[x]_{A_i}:\ S\subseteq X\right\}$ for each $i\in [n]$. It can be seen easily that $\mathcal A_1, \ldots, \mathcal A_n$ are algebras. For each of the at least $\vv(n)$ elements $x\in \ker(A_i)$ it holds that $\{x\}\in {\mathcal P}(X)\setminus \mathcal A_i$ . In particular, by the definition of $\vv(n)$, we find a family $\{U_i^1,U_i^2\}_{i\in [n]}$ such that if $Q\in \mathcal P(X)$ and $Q$ contains one of the two sets $U_i^1$ and $U_i^2$ and its intersection with the other one is empty, then $Q\notin \mathcal A_i$. For every $i\in [n]$, we now choose $Q_i\in {\mathcal A}_i$ to be the inclusion minimal set satisfying $U_i^1\subseteq Q_i$, and we note that $U_i^2\cap Q_i\neq \emptyset$ is implied. By the minimality of $Q_i$, it turns out that 
every equivalence class of $A_i$ that is contained in $Q_i$ needs to intersect $U_i^1$, and thus there is at least one such class $[z_i]_{A_i}$ intersecting both $U_i^1$ and $U_i^2$. Choosing arbitrary elements $x_i\in [z_i]_{A_i}\cap U_i^1$ and $y_i\in [z_i]_{A_i}\cap U_i^2$ for every $i\in [n]$ finally leads to a rainbow matching as desired.

To prove $\vv(n)\leq \vv_1(n)$, we need to argue that for every algebras $\mathcal A_1, \ldots, \mathcal A_n$ on a set $X$ with at least 
$\vv_1(n)$ pairwise disjoint sets in $\mathcal P(X) \setminus \mathcal A_i$, for each $i\in [n]$, there is a family $\{U_i^1,U_i^2\}_{i\in [n]}$ as described earlier. To do so, for each $i\in [n]$, we define equivalence relations $A_i$ on $X$ the equivalence classes of which are the inclusion minimal sets in $\mathcal A_i$. As, by the properties of an algebra, for every set $B\in \mathcal P(X)\setminus \mathcal A_i$ there is at least one element $b\in B$ with $\{b\}\notin \mathcal A_i$,
we conclude $|\ker(A_i)|\geq \vv_1(n)$, for every $i\in [n]$. Thus, by definition of $\vv_1(n)$,
we find a rainbow matching $x_1,y_1,\ldots,x_n,y_n$ as described above. Now, for every $i\in [n]$, let $U_i^1:=\{x_i\}$ and
$U_i^2:=\{y_i\}$. Then, whenever $U_i^j\subseteq Q$ holds for some $Q\in \mathcal P(X)$ and $j\in \{1,2\}$ we obtain $[y_i]_{A_i}=[x_i]_{A_i}\subseteq Q$, by definition of $A_i$, and thus $Q\cap U_i^{3-j}\neq \emptyset$. 
\end{appendix}

\end{document}